\documentclass[12pt]{article}
\usepackage[latin1]{inputenc}
\usepackage{color}
\usepackage{bbm}
\usepackage{amsmath,amssymb}
\usepackage{amsthm}
\usepackage{enumerate}
\usepackage{mathrsfs}
\usepackage{verbatim}
\usepackage{graphicx}
\usepackage{titling}
\predate{}
\postdate{}

\newtheorem{thm}{Theorem}

\newtheorem{lem}[thm]{Lemma}
\newtheorem{prop}[thm]{Proposition}
\newtheorem{defin}[thm]{Definition}

\newtheorem{rem}[thm]{Remark}

\makeatletter
\newcommand{\rd}{\,\mathrm{d}}

\newcommand{\bsx}{\boldsymbol{x}}
\newcommand{\bsy}{\boldsymbol{y}}
\newcommand{\bsz}{\boldsymbol{z}}

\newcommand{\bst}{\boldsymbol{t}}

\newcommand{\bszero}{\boldsymbol{0}}

\newcommand{\NN}{\mathbb{N}}

\newcommand{\ZZ}{\mathbb{Z}}

\newcommand{\disp}{\mathrm{disp}}

\newcommand{\cP}{\mathcal{P}}

\allowdisplaybreaks

\begin{document}
\author{Ralph Kritzinger and Jaspar Wiart\thanks{The authors are supported by the Austrian Science Fund (FWF),
Project F5509-N26, which is part of the Special Research Program ``Quasi-Monte Carlo Methods: Theory and Applications''.}}

\title{Improved dispersion bounds for \\ modified Fibonacci lattices}
\date{}
\maketitle 

\begin{abstract}
     We study the dispersion of point sets in the unit square; i.e. the size of the largest axes-parallel box amidst such point sets. It is known that  $
    \liminf_{N\to\infty} N\disp(N,2)\in \left[\frac54,2\right],
$
where $\disp(N,2)$ is the minimal possible dispersion for an $N$-element point set in the unit square. The upper bound 2 is obtained by an explicit point construction - the well-known Fibonacci lattice. In this paper we find a modification of this point set such that its dispersion is significantly lower than the dispersion of the Fibonacci lattice. Our main result will imply that $\liminf_{N\to\infty} N\disp(N,2)\leq \varphi^3/\sqrt{5}=1.894427...$
\end{abstract}

\section{Introduction and main results}
We consider point sets $\cP$ in the unit square $[0,1]^2$ consisting of $N\geq 1$ (not necessarily distinct) elements. We are interested in the size of the largest box amidst such point sets which does not contain any points of $\cP$. We speak of this size as the \emph{(standard) dispersion} of the point set $\cP$. More formally, we introduce the set $\mathcal{B}$ of all axes-parallel boxes in the unit square; i.e.  $\mathcal{B}=\{[x_1,y_1)\times [x_2,y_2)\mid 0\leq x_1\leq y_1\leq 1, 0\leq x_2\leq y_2\leq 1\}. $ The dispersion of $\cP$ can then be defined as $\disp(\cP):=\sup_{B\in\mathcal{B}, B\cap\cP=\emptyset}\lambda(B)$, where $\lambda(B)$ denotes the area of the box $B$.\\
In several recent papers the dispersion of point sets has also been considered with respect to periodic boxes on the torus (e.g. \cite{Bren, Ullr}). We identify the two-dimensional torus $\mathbb{T}([0,1]^2)$ with $[0,1]^2$ and introduce the set $\mathcal{B}_{\mathbb{T}}$ of periodic intervals on the torus $\mathbb{T}([0,1]^2)$ as follows. For $x,y\in [0,1]$ set
$$ I(x,y)=\begin{cases}
           [x,y) & \text{if $x\leq y$}, \\
           [0,y)\cup [x,1)& \text{if $x>y$,}
          \end{cases}$$
and for $\bsx=(x_1,x_2),\bsy=(y_1,y_2) \in [0,1]^2$ we set $B(\bsx,\bsy)=I(x_1,y_1)\times I(x_2,y_2)$.
We define $\mathcal{B}_{\mathbb{T}}:=\{B(\bsx,\bsy)\mid \bsx,\bsy \in [0,1]^2\}$ and the \emph{torus dispersion} as $\mathrm{disp}(\cP)_{\mathbb{T}}:=\sup_{B'\in\mathcal{B}_{\mathbb{T}}, B'\cap\cP=\emptyset}\lambda(B')$. \\
In the following we summarize relevant facts on the standard and periodic dispersion of point sets in the unit square. By~\cite[Theorem 1]{Dumi} we have \begin{equation} \label{lower54}
    \mathrm{disp}(\cP)\geq \max\left\{\frac{1}{N+1},\frac{5}{4(N+5)}\right\}
\end{equation}   for the standard dispersion of every $N$-element point set $\cP$ in the unit square.
On the other hand we know a construction of points with notably small dispersion - the well-known \emph{Fibonacci lattice}. In order to introduce this set, recall the definition of Fibonacci numbers: We set $F_1=F_2=1$ and $F_k=F_{k-1}+F_{k-2}$ for $k\geq 3$.
Let $m\geq 3$. The Fibonacci lattice $\mathcal{F}_m$ with $F_m$ elements is given by
$$ \mathcal{F}_m=\left\{\left(\frac{k}{F_m},\left\{\frac{kF_{m-2}}{F_m}\right\}\right): k=0,1,\dots,F_m-1\right\}, $$
where $\{x\}$ denotes the fractional part of a real number $x$.
By~\cite[Theorem 6.10]{Bren} we know that the dispersion of $\mathcal{F}_m$ for $m\geq 8$ is given by\footnote{The authors of~\cite{Bren} state this result for all $m\geq 6$, but it only holds for $m\geq 8$, as the inequality $F_5F_{m-2}\leq 2(F_m-1)$ which they use in their proof of Theorem 6.10 is only satisfied for $m\geq 8$ (with equality only for $m=8$.)} 
\begin{equation*} \label{fibodisp}
    \mathrm{disp}(\mathcal{F}_m)=\frac{2(F_m-1)}{F_m^2};
\end{equation*} 
i.e. $\lim_{m\to\infty}   |\mathcal{F}_m| \mathrm{disp}(\mathcal{F}_m)=2.$
Therefore we have 
\begin{equation} \label{constants}
    \liminf_{N\to\infty} N\disp(N,2)\in \left[\frac54,2\right],
\end{equation}
where $\disp(N,2):=\inf_{\cP\subset [0,1]^2:\, |\cP|=N}\disp(\cP)$. The aim of this paper is to narrow down this interval from above by finding a construction of points with smaller asymptotic dispersion than the Fibonacci lattice. \\
The outstanding role of the Fibonacci lattice in the theory of dispersion becomes clear when we consider the torus dispersion. By~\cite{Ullr} it is known that for every $N$-element point set $\cP$ in the unit square we have $\mathrm{disp}(\cP)_{\mathbb{T}}\geq \frac{2}{N}$ (see also~\cite[Theorem 6.1]{Bren}). Breneis and Hinrichs~\cite[Theorem 6.2]{Bren} could show that the torus dispersion of the Fibonacci lattice satisfies
  $$\mathrm{disp}(\mathcal{F}_m)_{\mathbb{T}}=\frac{2}{F_m},$$
which is best possible. In particular, these results show that $$\liminf_{N\to\infty} N\disp(N,2)_{\mathbb{T}}=2,$$ where $\disp(N,2)_{\mathbb{T}}:=\inf_{\cP\subset [0,1]^2:\, |\cP|=N}\disp(\cP)_{\mathbb{T}}$.

As already announced above, our goal is to narrow down the interval in~\eqref{fibodisp} by replacing the right-hand limit 2 by a smaller constant. To this end, we modify $\mathcal{F}_m$ in the following way:
\begin{defin}\label{def1}\rm
For a fixed Fibonacci number $F_m$ with $m\geq 3$ (since we always consider an arbitrary but fixed parameter $m$, we suppress the dependence of all these definitions on $m$), define the $F_m$-periodic function $\pi: \NN_0\to\{0,1,\dots,F_m-1\}$ by
\[ 
 k \mapsto F_m\left\{\frac{kF_{m-2}}{F_m}\right\}=
kF_{m-2} \pmod{F_m}.
\]
On the set $\{0,1,\dots,F_m-1\}$ the function $\pi$ operates as a permutation. We introduce the $F_m$-periodic map $s: \NN_0\to\{1,\varphi\}$ to be
\[
 i \mapsto 
\begin{cases}
    \varphi  & \text{if $\pi(i)<\pi(i+1)$}, \\
    1        & \text{otherwise,}
\end{cases} 
\]
where $\varphi=\frac{\sqrt{5}+1}{2}$ is the golden ratio. Let $L:=\sum_{i=0}^{F_m-1}s(i)$ (we will see in Lemma ~\ref{sumofgaps} that $L=\varphi^{m-1}$) and $x_k=\sum_{i=0}^{k-1}s(i)$ for $k=0,\dots,F_m-1$.  We define the \emph{modified Fibonacci lattice} $\widetilde{\mathcal{F}}_m$ to be the set of points
   \[
   \widetilde{\mathcal{F}}_m:=\left\{\left(\frac{x_k}{L},\frac{x_{\pi(k)}}{L}\right),k=0,1,\dots,F_m-1\right\}.
   \]
\end{defin}

\begin{figure}[h]  
     \centering
     {\includegraphics[width=70mm]{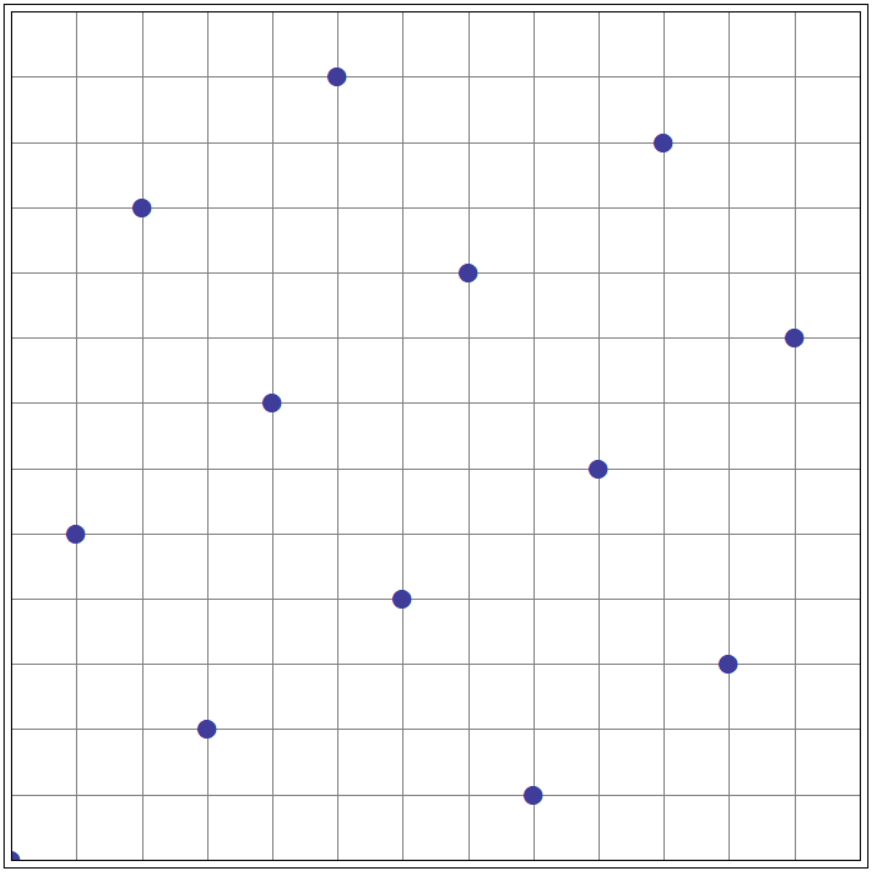}}
     \hspace{.1in}
     {\includegraphics[width=70mm]{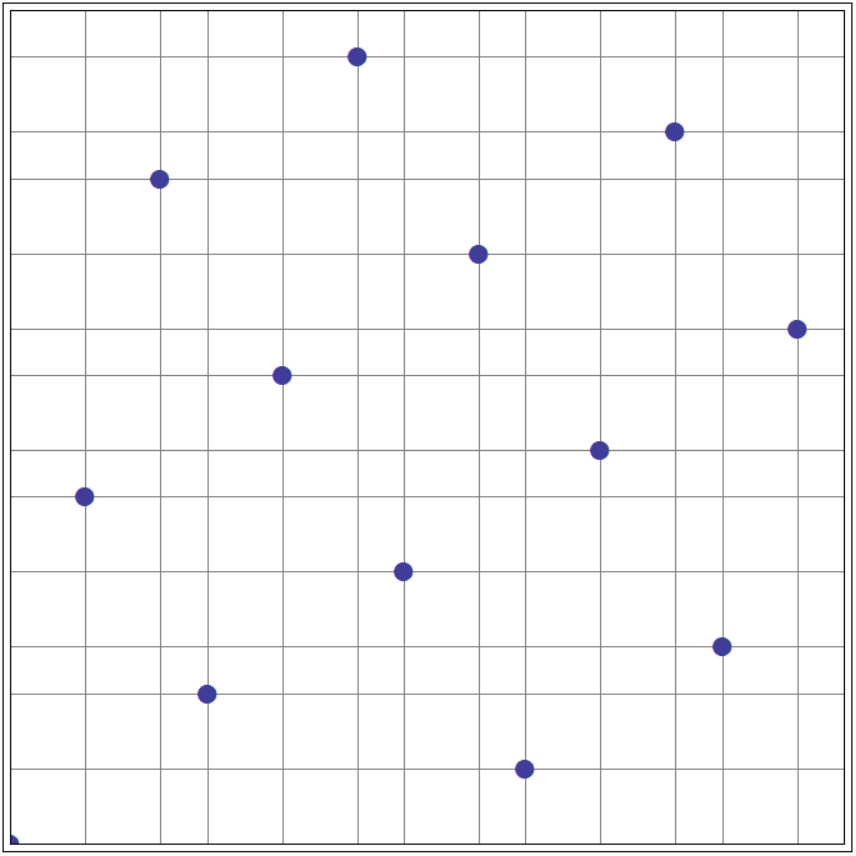}}
     \caption{Left: the points of $\mathcal{F}_7$. Right: The points of $\widetilde{\mathcal{F}}_7$. Notice how the gaps between the $x$-coordinates of the points vary in length and that the $y$-coordinates are a permutation of the $x$-coordinates.}
     \label{nolabel}
\end{figure}

\begin{rem} \rm
The point $(0,0)$ can be removed without changing the dispersion of $\widetilde{\mathcal F}_m$. However its inclusion simplifies those proofs which rely on the torus dispersion results of Fibonacci lattices in \cite{Bren}.
\end{rem}

A maximal empty box amidst a point set is one that is bounded on each side by either the edge of the unit square or a point from the point set. We define two different kinds of maximal empty boxes. 
  
\begin{defin} \rm
Given some set of points, a \emph{maximal interior box (amidst those points)} is an empty box that is bounded on each side by a point. A \emph{maximal exterior box (amidst those points)} is an empty box where at least one side is bounded by an edge of the unit square and each of the other three sides is bounded by either an edge of the unit square or a point. 
\end{defin}

The motivation behind $\tilde{\mathcal F}_m$ was to find a way to modify the Fibonacci lattice in such a way as to improve the constant of $2$ in the limit.  The worst maximal boxes amidst the points of the Fibonacci lattice are the tall and narrow boxes (i.e.\ boxes of width $2/F_m$) and the wide and short boxes (i.e.\ boxes of height $2/F_m$). The idea is to simply make the tall and narrow boxes a little bit more narrow by adjusting the gaps between the $x$-coordinates of the points and then adjusting the $y$-coordinates by the same rule. A simple way to do this is to make the $x$-axis gap between the points $(k/F_m,\{kF_{m-2}/F_m\})$ and $((k+1)/F_m,\{(k+1)F_{m-2}/F_m\})$ smaller if $(k/F_m,\{kF_{m-2}/F_m\})$ is the top of tall and narrow box. This gives rise to Definition \ref{def1} with the values of $s$ determined by solving for the values that makes the area of all interior boxes equal.


\begin{figure}[h] 
     \centering
     {\includegraphics[width=70mm]{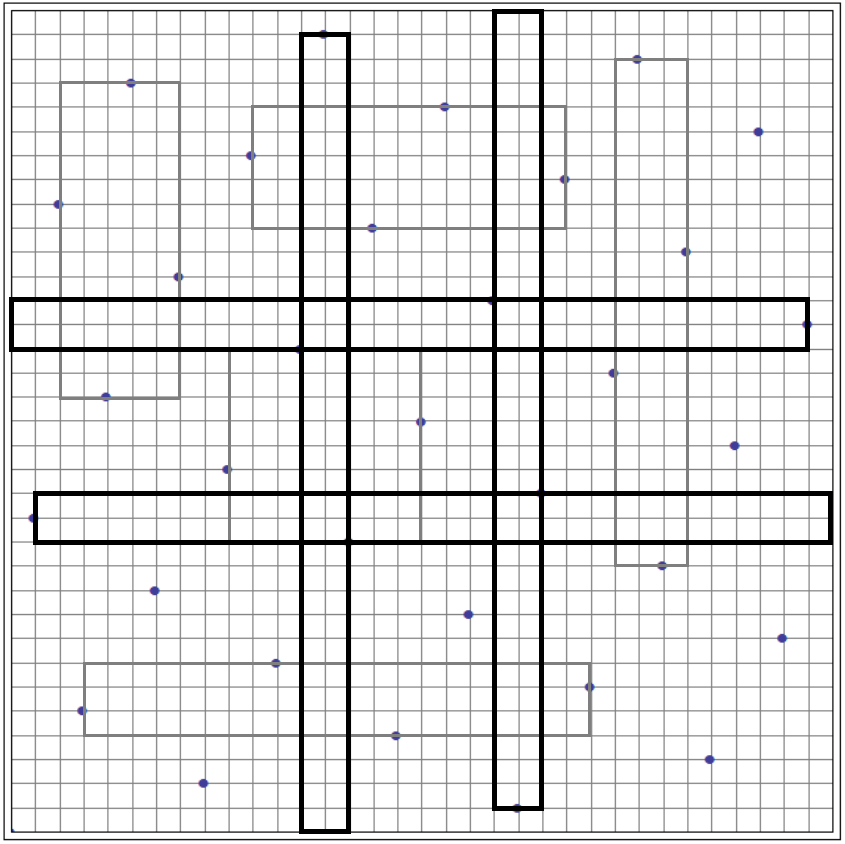}}
     \hspace{.1in}
     {\includegraphics[width=70mm]{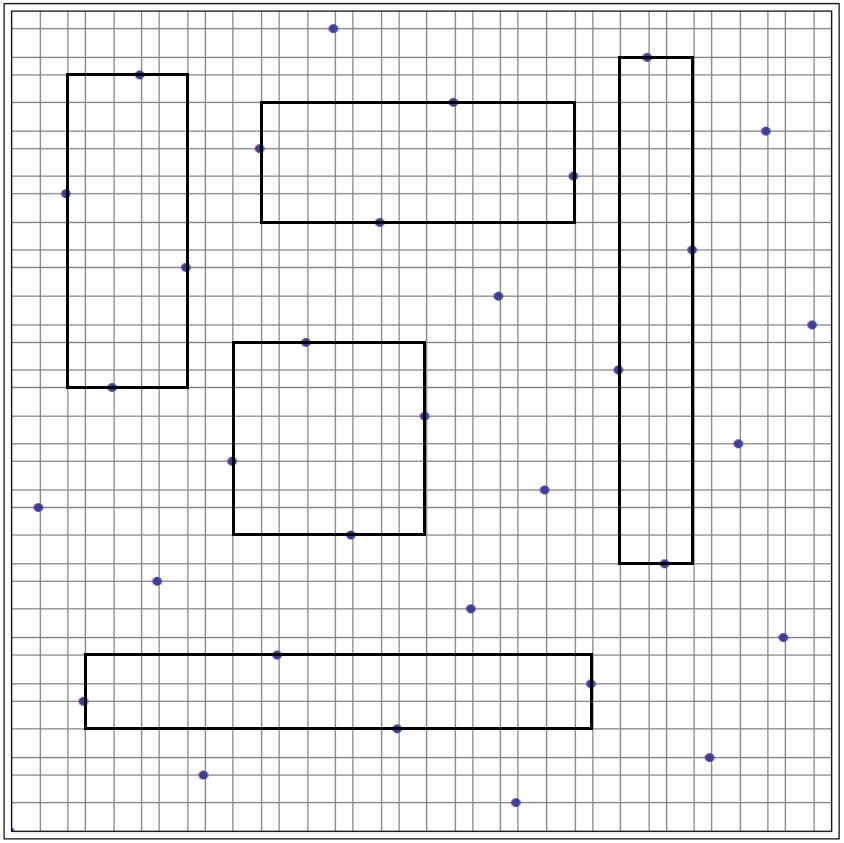}}
     \caption{Left: The maximal interior boxes (grey) of the Fibonacci lattice have dimensions $(F_{k}/F_m)\times (F_{m-k+3}/F_m)$ for $k=4,\dots, m-1$ (here $m=9$). The area of every one of these boxes is strictly less than $\mathrm{disp}({\mathcal{F}}_m)=2(F_m-1)/F_m^2$; i.e. the area of the black boxes, which are the only four boxes in $\mathcal{F}_m$ with maximal size.  Right: The maximal interior boxes of the modified Fibonacci lattice have dimensions $(\varphi^{k-1}/\varphi^{m-1})\times(\varphi^{m-k+2}/\varphi^{m-1})$ for $k=4,\dots,m-1$, each of which has area equal to $\mathrm{disp}(\widetilde{\mathcal{F}}_m)=\varphi^{3-m}$. Notice that there are far more than $4$ maximal empty boxes with area equal to $\varphi^{3-m}$ amidst the points of $\widetilde{\mathcal{F}}_m$.}
     \label{boxes}
\end{figure}

The central result of this paper is the following theorem on the dispersion of $\widetilde{\mathcal F}_m$.

\begin{thm} \label{precise}
Let $m\geq 5$. Every maximal interior box amidst $\widetilde{\mathcal F}_m$ has area $\varphi^{3-m}$ while the maximal exterior boxes have area less than $\varphi^{3-m}$. In particular $\mathrm{disp}(\widetilde{\mathcal F}_m)=\varphi^{3-m}$ and $$\lim_{m\to\infty}|\widetilde{\mathcal{F}}_m|\mathrm{disp}(\widetilde{\mathcal{F}}_m)=\frac{\varphi^3}{\sqrt{5}}=1+\frac{2}{\sqrt{5}}=1.894427...  $$
\end{thm}

\begin{rem}
The paper ``How to find a battleship'' \cite{Fiat}  uses Fibonacci lattices to find ``battleships'' on the sea $[-F_m,F_m]\subseteq \ZZ^2$.  They found that Fibonacci lattices can find any battleships of size $\varphi^3/\sqrt{5}F_m$. However, their definition of battleship is the subset of $\ZZ^2$ that is contained in an axis-parallel box and the area of the battleship is the number of integer points that it contains. With this definition of battleship they can ignore the boxes that cause the dispersion of Fibonacci lattices to be $2(F_m-1)/F_m$.
\end{rem}

 Before we continue to prove this result in the following two sections, we would like to add several words on recent developments in the theory of dispersion of high-dimensional point sets. For point sets in $[0,1]^d$, where $d\geq 3$, we search for the size of the largest axes-parallel box $B\subseteq [0,1]^d$ which does not contain any point, which we call again the dispersion of the point set $\cP$ and we write $\disp(\cP)$. It is well-established that the optimal order of dispersion for an $N$-element point set in $[0,1]^d$ with respect to $N$ is $\mathcal{O}(N^{-1})$ like in the two-dimensional case. More precisely it is known that $\disp(\cP)\geq \frac{c_d}{N}$ for all $N$-element point sets in the $d$-dimensional unit cube $[0,1]^d$, where $c_d>0$ is independent of $N$ and tends to infinity at least logarithmically with $d$ (see~\cite[Theorem 1]{Aist}). Note that in the torus setting the constant in the corresponding lower bound grows linearly, see~\cite{Ullr}. On the other hand, we have the upper bound $\disp(\cP)\leq \frac{2^{7d}}{N}$, which is attained for certain $(t,m,d)$-nets (see~\cite[Section 4]{Aist}). A similar upper bound on multidimensional Hammersley point sets was earlier obtained in~\cite{Rote}, where the constant is super-exponential in the dimension $d$ though. The logarithmic growth of $c_d$ in the lower bound from~\cite{Aist} was shown to be best possible when we allow a worse dependence on $N$ (see~\cite{Lit,Sos,Ullr2} for existence results and~\cite{Lin,Ullr3} for deterministic approaches, where the former article uses a different language). 
 Other recent papers which study the dispersion of high-dimensional point sets, among others, are~\cite{Krieg, rudo}. \\
 In Section~\ref{aux} we collect several auxiliary results, which we use in Section~\ref{proof} to prove Theorem~\ref{precise}. In Section~\ref{improve} we discuss a slightly different modification of the Fibonacci lattice which has slightly lower dispersion than $\widetilde{\mathcal{F}}_m$ and perform numerical experiments on the $L_2$ discrepancy of these new point sets. In the final Section~\ref{conclusion} we summarize our main results, discuss unanswered questions and mention a result similar to Theorem~\ref{precise}.

\section{Auxiliary results} \label{aux}

We begin by recalling some identities involving the golden ratio. First is $\varphi^2=\varphi+1$. For $m\geq 2$ we have the beautiful relation
  \begin{equation} \label{superiden}
      \varphi^m=\varphi F_m+F_{m-1}
  \end{equation}
  between Fibonacci numbers and powers of the golden ratio, which can be easily shown by induction on $m$. We will use this relation several times throughout this paper. We always assume $m\geq 5$ in the following.

\begin{lem} \label{sumofgaps}
    For $L$ as introduced in Definition~\ref{def1} we have $
        L=\varphi^{m-1}.
    $
\end{lem}

\begin{proof}
   The essential observation for the proof of this lemma is the following:  For $\ell\in\{0,1,\dots,F_{m-2}\}$ define the disjoint sets $A_{\ell}=\{n+\ell F_m:n=0,1,\dots,F_m-1\}$. The inequality $\pi(i)<\pi(i+1)$ is equivalent to the existence of an integer $\ell\in\{0,1,\dots,F_{m-2}-1\}$ such that $iF_{m-2}$ and $ (i+1)F_{m-2}$ are both in $A_{\ell}$. On the other hand, $\pi(i)>\pi(i+1)$ is equivalent to the existence of an integer $\ell\in\{0,1,\dots,F_{m-2}-2\}$ such that $iF_{m-2} \in A_{\ell}$, whereas $(i+1)F_{m-2}\in A_{\ell+1}$.\\
  Let $|\{i\in\{0,\dots,F_m-1\}\}: s(i)=1\}|=j$. Since $0 F_{m-2} \in A_0$, this assumptions implies $F_mF_{m-2}\in A_j$. We also have the trivial inequalities
    $$F_{m-2}F_{m}\leq F_mF_{m-2}<(F_{m-2}+1)F_{m}; $$
    i.e. $F_mF_{m-2}\in A_{F_{m-2}}$. This implies $j=F_{m-2}$, which yields
    $$L=\sum_{i=0}^{F_m-1}s(i)=(F_m-F_{m-2})\varphi+F_{m-2}=F_{m-1}\varphi+F_{m-2}=\varphi^{m-1}$$
    by relation~\eqref{superiden}.
\end{proof}

\begin{lem} \label{translated} 
  Let $3\leq k\leq m-1$ and $r\in\NN_0$. If $k$ is odd, then
  $$ \sum_{i=r}^{r+F_k-1}s(i)=\begin{cases}
                   \varphi^{k-1}+\varphi^{-1} & \text{if $\pi(r)<F_{m-k}$}, \\
                   \varphi^{k-1} & \text{otherwise.} \\
                              \end{cases}$$      
    If $k$ is even, then
  $$ \sum_{i=r}^{r+F_k-1}s(i)=\begin{cases}
                   \varphi^{k-1}-\varphi^{-1} & \text{if $\pi(r)\geq F_m-F_{m-k}$}, \\
                   \varphi^{k-1} & \text{otherwise.} \\
                              \end{cases}$$
                
\end{lem}
\begin{proof}
 First we observe that 
   \begin{equation} \label{relation}
       F_kF_{m-2}-F_{k-2}F_m=(-1)^k F_{m-k}.
   \end{equation} 
   This follows from the well-known identity $F_pF_{q+1}-F_qF_{p+1}=(-1)^qF_{p-q}$ for nonnegative integers $p\geq q$, which yields
   $$ F_{k-2}F_m-F_kF_{m-2}=(F_k-F_{k-1})F_m-F_k(F_m-F_{m-1})=F_{m-1}F_{k}-F_{k-1}F_{m}=(-1)^{k-1}F_{m-k}.$$
    Let us assume that $|\{i\in\{r,\dots,r+F_k-1\}\}: s(i)=1\}|=j$ and that $r F_{m-2} \in A_{\eta}$; i.e.
   \begin{equation*}
       \eta F_{m} \leq r F_{m-2} < (\eta+1)F_{m},
   \end{equation*}
   or equivalently
      \begin{equation} \label{etabounds}
       0 \leq r F_{m-2}-\eta F_{m} < F_{m}.
   \end{equation}
  With the arguments as in the proof of the previous lemma these assumptions imply $(F_k+r)F_{m-2}\in A_{\eta+j}$; i.e.
   $$ (\eta+j)F_m \leq (F_k+r)F_{m-2} < (\eta+j+1)F_m.$$
   We solve these inequalities for $jF_m$ and obtain
   \begin{equation} \label{ineqj}
        (F_k+r)F_{m-2}-(\eta+1)F_m < j F_m \leq (F_k+r)F_{m-2}-\eta F_m.
   \end{equation} 
   Using~\eqref{relation} and~\eqref{etabounds}, for odd $k$ we can bound the right-hand-side of~\eqref{ineqj} from above by
    $$ F_mF_{k-2}-F_{m-k}+F_{m}<(F_{k-2}+1)F_m $$
   and the left-hand-side of~\eqref{ineqj} from below by
   $$ F_mF_{k-2}-F_{m-k}-F_m>(F_{k-2}-2)F_m. $$
   This yields $j=F_{k-2}-1$ or $j=F_{k-2}$. 
     If $k$ is an even number, then 
  we can bound the right-hand-side of~\eqref{ineqj} from above by
   $$ F_mF_{k-2}+F_{m-k}+F_{m}<(F_{k-2}+2)F_m $$
   and the left-hand-side of~\eqref{ineqj} from below by
   $$ F_mF_{k-2}+F_{m-k}-F_m>(F_{k-2}-1)F_m. $$
   This yields $j=F_{k-2}$ or $j=F_{k-2}+1$. 
   Hence if $k$ is odd we either have
   $$ \sum_{i=r}^{r+F_k-1}s(i)=F_{k-1}\varphi+F_{k-2}=\varphi^{k-1} $$
   or 
   \begin{equation} \label{oddk}
       \sum_{i=r}^{r+F_k-1}s(i)=(F_{k-1}+1)\varphi+(F_{k-2}-1)=\varphi^{k-1}+\varphi^{-1}, 
   \end{equation} 
   whereas for even $k$ we either have
   $$ \sum_{i=r}^{r+F_k-1}s(i)=F_{k-1}\varphi+F_{k-2}=\varphi^{k-1} $$
   or 
   $$ \sum_{i=r}^{r+F_k-1}s(i)=(F_{k-1}-1)\varphi+(F_{k-2}+1)=\varphi^{k-1}-\varphi^{-1}. $$
  Next we investigate for which $r$ we get the larger result~\eqref{oddk} for odd $k$. This happens if and only if $(F_k+r)F_{m-2}\in A_{\eta+F_{k-2}-1}$; i.e. if
  $$ (F_k+r)F_{m-2}< (\eta+F_{k-2})F_m. $$
  Using relation~\eqref{relation} we find the equivalent inequality
  $$ rF_{m-2}-\eta F_m<F_{m-k}. $$
  The fact that $\eta=\lfloor\frac{rF_{m-2}}{F_m} \rfloor$ yields the equivalent inequality
  $$ F_m\left\{\frac{rF_{m-2}}{F_m}\right\} < F_{m-k}, $$
 i.e. $\pi(r)< F_{m-k}$. \\
  Similarly we show for which $r$ we have $\sum_{i=r}^{r+F_k-1}s(i)=\varphi^{k-1}-\varphi^{-1}$ in the case of even $k$. This is the case if and only if $(F_k+r)F_{m-2}\in A_{\eta+F_{k-2}+1}$; i.e. if
  $$ (F_k+r)F_{m-2}\geq (\eta+F_{k-2}+1)F_m. $$
  With relation~\eqref{relation} we transform this inequality into
  $$ rF_{m-2}-\eta F_m\geq F_m-F_{m-k}; $$
  i.e. $\pi(r)\geq F_m- F_{m-k}$.
\end{proof}

\begin{lem} \label{pointslow} 
If $3\leq k \leq m-1$ is odd and $r\in\NN_0$ such that $\pi(r)<F_{m-k}$, then we have $\pi(r+F_k)=\pi(r)+F_m-F_{m-k}$. \\
If $3\leq k \leq m-1$ is even and $r\in\NN_0$  such that $\pi(r)\geq F_m-F_{m-k}$, then we have $\pi(r+F_k)=\pi(r)-(F_m-F_{m-k})$.
\end{lem}

\begin{proof}
 We consider odd numbers $k$. The inequality $\pi(r)<F_{m-k}$ is equivalent to
     \begin{equation}
         \label{rbound} \left\{\frac{rF_{m-2}}{F_m}\right\}<\frac{F_{m-k}}{F_m}.
     \end{equation} 
     We employ~\eqref{relation} once again to obtain
     \begin{align*}
         \pi(r+F_k)=& F_m  \left\{\frac{(r+F_k)F_{m-2}}{F_m}\right\}=F_m  \left\{\frac{rF_{m-2}+F_{m}F_{k-2}-F_{m-k}}{F_m}\right\} \\
         =& F_m  \left\{\frac{rF_{m-2}-F_{m-k}}{F_m}\right\}=F_m  \left\{\left\{\frac{rF_{m-2}}{F_m}\right\}-\frac{F_{m-k}}{F_m}\right\},
     \end{align*}
     since the fractional part does not change by adding or removing integers. Since
     $$ \left\{\frac{rF_{m-2}}{F_m}\right\}-\frac{F_{m-k}}{F_m}<0 $$
     by~\eqref{rbound} and
     $$  \left\{\frac{rF_{m-2}}{F_m}\right\}-\frac{F_{m-k}}{F_m}\geq -\frac{F_{m-k}}{F_m}>-1 $$
     we have 
     $$  \left\{\frac{rF_{m-2}}{F_m}\right\}-\frac{F_{m-k}}{F_m} \in (-1,0). $$
     For real numbers $x\in (-1,0)$ we clearly have $\{x\}=x+1$, which yields
     $$\pi(r+F_k)=F_m\left(\left\{\frac{rF_{m-2}}{F_m}\right\}-\frac{F_{m-k}}{F_m}+1\right)=\pi(r)+F_m-F_{m-k}  $$
     and the first part of the lemma is verified. The proof of the second part is similar.
     \end{proof}

\section{Proof of Theorem~\ref{precise}} \label{proof}

First we adapt the proof in \cite{Bren} to get the following proposition:

\begin{prop} \label{boxshape}
The maximal periodic boxes amidst the points of $\widetilde{\mathcal F}_m$ are of the form
\[
(x_r,x_{r+F_k})\times (x_s,x_{s+F_{m-k+3}})
\]
for $k=3,\dots,m$ and certain integers $r,s\geq 0$ and where the indices are taken modulo $F_m$.
\end{prop}

\begin{proof}
This has essentially been proven in~\cite{Bren} but has not been stated there explicitly. We show how to derive this result from~\cite[Lemma 6.8]{Bren}. We consider the function
  $\pi$ from Definition~\ref{def1}. Obviously $\pi$ is $F_m$-periodic; i.e. $\pi(k+F_m)=\pi(k)$ for every $k\geq 0$. For integers $1\leq \ell\leq F_m$ and $t\geq 0$ we define sequences
   $$ Y_{\ell,t}=(\pi(k))_{k=t}^{t+\ell-1} $$
   consisting of $\ell$ distinct integers in $\{0,1,\dots,F_m-1\}$.
   Let us consider a set $Y_{\ell,0}$. We order the $\ell$ numbers in this set from the lowest to the highest. We define the distance between two integers $a,b\in\{0,1,\dots,F_m-1\}$ by $d(a,b)=b-a$ if $a\leq b$ and $d(a,b)=F_m+a-b$ if $a>b$. Then it is known that there occur at most three different distances between two consecutive elements in $Y_{\ell,0}$, where we also consider the highest and the lowest number as consecutive with respect to the distance notion $d$. More precisely, Breneis and Hinrichs could prove that for $\ell=F_{k}-j$ with $1\leq j \leq F_{k-2}$ and $3\leq k \leq m$ these three distances are $F_{m-k+3}$, which occurs $j$ times, then $F_{m-k+2}$, which occurs $F_{k-1}-j$ times and finally $F_{m-k+1}$,  which occurs $F_{k-2}-j$ times. This follows by the equalities (3.1) and (3.3) in~\cite[Lemma 8]{Bren} by setting $n=F_m$ and $q=F_{m-2}$ and using relation~\eqref{relation}. The same distances occur in the same frequency between consecutive elements of $Y_{F_{k}-j,t}$ for any integer $t\geq 0$, since  $ d(\pi(a),\pi(b))=d(\pi(a+t),\pi(b+t)) $ for any $t\geq 0$; i.e. the distance between $\pi(a)$ and $\pi(b)$ is invariant with respect to equal translations of $a$ and $b$. The crucial implication is the following: For $\ell=F_k-1$ there exists still one gap in $Y_{\ell,t}$ of length $F_{m-k+3}$, whereas for $\ell=F_k$ the largest gap is only $F_{m-k+2}$. Since there is an one-to-one correspondence between the maximal (periodic) boxes of $\mathcal{F}_m$ and those of $\widetilde{\mathcal{F}}_m$ the maximal empty periodic boxes amidst $\widetilde{\mathcal{F}}_m$ are of the form
     $ (x_a,x_{b})\times (x_{\pi(c)},x_{\pi(d)}) $
     for certain $a,b,c,d \in \{0,1,\dots,F_m-1\}$ such that $d(a,b)=F_k$ and $d(\pi(c),\pi(d))= F_{m-k+3}$. This yields the result.
   \end{proof}

We are ready to show Theorem~\ref{precise}.

\begin{proof}[Proof of Theorem~\ref{precise}]
 Consider a maximal periodic empty box amidst the points of $\widetilde{\mathcal F}_m$, which by Proposition~\ref{boxshape} is of the form
$
(x_{r},x_{r+F_k})\times (x_{s},x_{s+F_{m-k+3}})
$
for some $k=3,\dots,m$ and integers $r,s\geq 0$ and where the indices are taken modulo $F_m$. By Lemma~\ref{translated} we know that in many cases the (torus) distance of $x_{r}$ and $x_{r+F_k}$ is $\varphi^{k-1}/L$ and the distance of $x_{s}$ and $x_{s+F_{m-k+3}}$ is $\varphi^{m-k+2}/L$. In this case the size of the chosen box is exactly
     $$ \frac{\varphi^{k-1}\varphi^{m-k+2}}{L^2}=\varphi^{3-m}. $$
The size of the box $(x_{r},x_{r+F_k})\times (x_{s},x_{s+F_{m-k+3}})$ can only have a different size if $k$ is odd and $\pi(r)<F_{m-k}$ or if $k$ is even and $\pi(r)\geq F_m-F_{m-k}$ or if similar instances apply for $s$ according to Lemma~\ref{translated}. We will show in the following that such boxes are necessarily periodic boxes wrapping around one side of the unit square, which yields that all maximal interior boxes amidst $\widetilde{\mathcal F}_m$ have volume $\varphi^{3-m}$.  Due to certain symmetries of  $\widetilde{\mathcal{F}}_m$ it is enough to only consider the cases where the (torus) distance of $x_{r}$ and $x_{r+F_k}$ is different from $\varphi^{k-1}/L$. These symmetries as explained in the following imply that for every maximal empty periodic box amidst $\widetilde{\mathcal{F}}_m$ that wraps around the $x$-axis there is exactly one periodic box of same size that wraps around the $y$-axis.
 \begin{itemize}
     \item Let $m$ be even. Then $(x_k,x_l)\in \widetilde{\mathcal{F}}_m$ implies that $(x_{l},x_k)\in \widetilde{\mathcal{F}}_m$; i.e. $\widetilde{\mathcal{F}}_m$ is symmetric with respect to the first median.
     \item Let $m$ be odd. Then $(x_k,x_{l})\in \widetilde{\mathcal{F}}_m\setminus\{0\}$ implies that $(x_{l},x_{F_m-k})\in \widetilde{\mathcal{F}}_m\setminus\{0\}$. Geometrically that means than for odd $m$ the point set $\widetilde{\mathcal{F}}_m\setminus\{0\}$ is invariant with respect to a counter-clockwise quarter rotation of the unit square and a subsequent translation of all points a distance $\varphi^{-1}$ in positive $x$-direction.
 \end{itemize}
 The properties of $\widetilde{\mathcal{F}}_m$ as explained in the previous two points can  be seen as follows: The condition $(x_k,x_l)\in \widetilde{\mathcal{F}}_m$ yields $l=\pi(k)$. Therefore $(x_{l},x_k)\in \widetilde{\mathcal{F}}_m$ can only hold if $\pi(\pi(k))=k$. Hence we must show that $\pi(\pi(k))=k$ for all $k\in\{0,1,\dots,F_m-1\}$ if $m$ is even, which is equivalent to $F_{m-2}^2 \equiv 1 \pmod {F_m}$. The validity of this congruence follows from the identity
 $$ F_{m-2}^2+F_m^2-3F_{m}F_{m-2}=(-1)^m, $$
 which holds for all $m\geq 3$ and can be shown by induction on $m$. From the same identity we derive
 $\pi(\pi(k))=F_m-k$ for all $k\in\{1,\dots,F_m-1\}$ if $m$ is odd, which yields that $(x_k,x_{l})\in \widetilde{\mathcal{F}}_m\setminus\{0\}$ implies that $(x_{l},x_{F_m-k})\in \widetilde{\mathcal{F}}_m\setminus\{0\}$ for odd $m$. Thus we are left to verify that boxes with width not equal to $\varphi^{k-1}$ are exterior boxes and have area  smaller than $\varphi^{3-m}$.
  \begin{figure}[h] 
     \centering
     {\includegraphics[width=70mm]{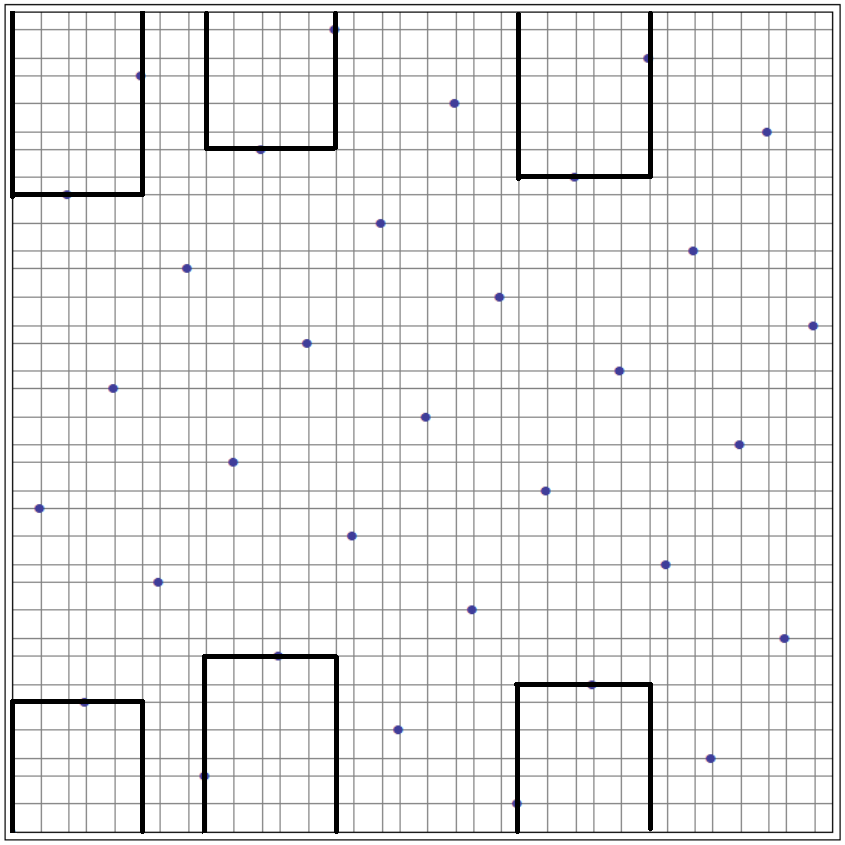}}
     \hspace{.1in}
     {\includegraphics[width=70mm]{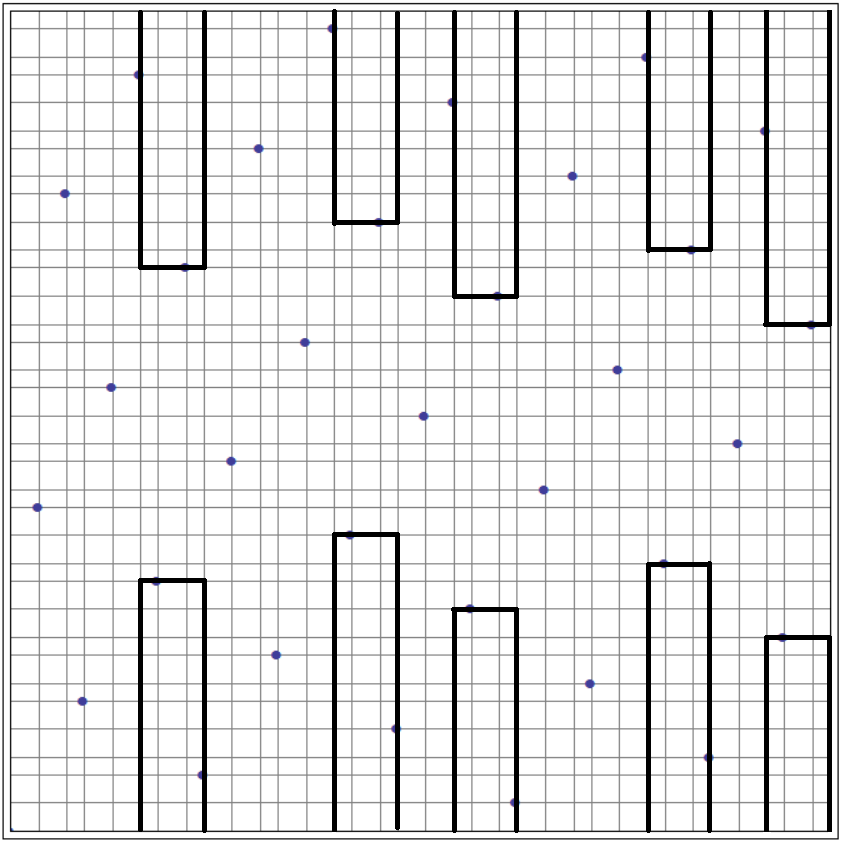}}
     \caption{Left: Maximal empty boxes amidst $\widetilde{\mathcal{F}}_9$ of width $\varphi^{k-1}+\varphi^{-1}$ (here fore $k=5$) are necessarily strict periodic boxes. Right: The same applies for maximal empty boxes of width $\varphi^{k-1}-\varphi^{-1}$ (here for $k=4$). For every periodic box drawn in the picture there is a corresponding periodic box of the same size that wraps around the left and right sides of the unit square.}
     \label{periodic}
\end{figure}

Let us first assume that $\pi(r)< F_{m-k}$ for an odd $k>3$ (we will treat the case $k=3$ below) such that the (torus) distance of $x_{r}$ and $x_{r+F_k}$ is $(\varphi^{k-1}+\varphi^{-1})/L$. Then by Lemma~\ref{pointslow} we know $\pi(r+F_k)=\pi(r)+F_m-F_{m-k}$. Therefore between the $y$-coordinates of the points $(x_{r},x_{\pi(r)})$ and $(x_{r+F_k},x_{\pi(r+F_k)})$ which bound the box from the left and right side, respectively, we have
 $F_m-F_{m-k}>F_{m-k+3}$ steps. (By steps we understand distances of length $s(i)/L$ for certain consecutive indices $i\in\NN_0$.) The latter inequality holds because
 $$ F_{m-k+3}+F_{m-k}\leq F_{m-1}+F_{m-4}<F_{m-1}+F_{m-2}=F_m.$$  This distance of more than $F_{m-k+3}$ steps is too large for the box to be a non-periodic box; so boxes of the form $(x_{r},x_{r+F_k})\times (x_{s},x_{s+F_{m-k+3}})$ where $\pi(r)<F_{m-k}$ and $k$ is odd are strict periodic boxes wrapping around the $x$-axis. We need to estimate the (non-periodic) parts of this box at the bottom and at the top (see Figure~\ref{periodic}). We consider the sequence $Y=\{x_{\pi(r+1)},\dots,x_{\pi(r+F_k-1)}\}$, which are the $y$-coordinates of the $F_k-1$ points of $\widetilde{\mathcal{F}}_m$ whose $x$-coordinates are $\{x_{r+1},\dots,x_{F_k-1}\}$; i.e. of the points between the points which bound the empty box we consider from the left and right, respectively. By the proof of Proposition~\ref{boxshape} we know that the largest distance between consecutive elements in the nondecreasing reordering of $\{\pi(r+1),\dots,\pi(r+F_k-1)\}$ is $F_{m-k+3}$, which occurs exactly once. If we now include the the index of the $y$-coordinate $x_{\pi(r)}$ of the point which bounds the box from the left side to this set, this largest distance of length $F_{m-k+3}$ is separated into two distances of lengths $F_{m-k+2}$ and $F_{m-k+1}$, respectively. The same applies if we include $\pi(r+F_k)$, i.e. the index of the $y$-coordinate of the point which bounds the box from the right, to the set $\{\pi(r+1),\dots,\pi(r+F_k-1)\}$. But that means that both separated parts of the periodic box can extend at most $F_{m-k+2}$ steps in height, because each part extends at least $F_{m-k+1}$ steps in height. By Lemma~\ref{translated}, the height of both separated parts of the periodic box is at most $(\varphi^{m-k+1}+\varphi^{-1})/L$. Hence the area of both parts is at most
 $$ \frac{(\varphi^{k-1}+\varphi^{-1})(\varphi^{m-k+1}+\varphi^{-1})}{L^2}, $$
 respectively.
 This expression is smaller than $\varphi^{3-m}$ since for $k\in\{4,\dots,m-1\}$ we have
 \begin{align*}
  (\varphi^{k-1}+\varphi^{-1})(\varphi^{m-k+1}+\varphi^{-1}) =& \varphi^{m}+\varphi^{k-2}+\varphi^{m-k}+\varphi^{-2}  \\ 
  \leq& \varphi^{m}+\varphi^{m-3}+\varphi^{m-4}+\varphi^{-2}
  \\ =&\varphi^{m}+\varphi^{m-2}+\varphi^{-2}\\ <&\varphi^{m}+\varphi^{m-2}+\varphi^{m-3}= \varphi^{m+1}.
 \end{align*}
 Similarly we can show that maximal empty boxes with width $\varphi^{k-1}-\varphi^{-1}$ (i.e. where $k$ is even and $\pi(r)\geq F_m-F_{m-k}$) are strictly periodic and are separated into two non-periodic boxes of size less than $\varphi^{3-m}$. \\
 Finally we consider boxes of width $2\varphi/L$ (i.e. where $k=3$ and $\pi(r)< F_{m-3}$). Then with the same arguments as above we find that such a box is strictly periodic. It is bounded from above and below by the point $(x_{r+1},x_{\pi(r+1)})$. Since $x_{\pi(r+1)}=x_{\pi(r)+F_{m-2}}$, we find that the height of both non-periodic parts of the box is at most $\varphi^{m-2}/L$ and their areas are smaller than $\varphi^{3-m}$. 
 That completes the proof of the dispersion result. \\
 It remains to prove that $\lim_{m\to\infty}|\widetilde{\mathcal{F}}_m|\mathrm{disp}(\widetilde{\mathcal{F}}_m)=\frac{\varphi^3}{\sqrt{5}}$.  Since $\varphi^{m}=F_{m}\varphi+F_{m-1}$ and $\lim_{m\to\infty}\frac{F_{m-1}}{F_m}=\varphi^{-1}$, we obtain
$$ \lim_{m\to\infty}|\widetilde{\mathcal{F}}_m|\mathrm{disp}(\widetilde{\mathcal{F}}_m)=\varphi^3 \lim_{m\to\infty} \frac{F_m}{F_m\varphi+F_{m-1}}=\frac{\varphi^3}{\varphi+\varphi^{-1}}=\frac{\varphi^3}{\sqrt{5}}, $$
where we considered $\varphi+\varphi^{-1}=\sqrt{5}$.
\end{proof}

\section{Improvements in dispersion and discrepancy} \label{improve}

    We can improve the dispersion of $\widetilde{\mathcal{F}}_m$ even further by altering the values for $s(0)$ and $s(F_m-1)$.
    It seems reasonable to choose those gaps such that the size of the largest maximal exterior boxes matches the area of the interior boxes. Let $s(0)=s(F_m-1)=x$ for some positive value $x$ we specify below. We set $L=\varphi^{m-1}+2x-\varphi^2$ and $x_k$ for $k=0,1,\dots,F_m-1$ as in Definition~\ref{def1} and define $\widetilde{\mathcal{F}}_m'$ with respect to these choices of $L$ and $x_k$. For small $m$ we found numerically that $x$ must be chosen as follows:
    \begin{itemize}
        \item If $m$ is odd, choose $x$ such that those maximal empty boxes with three sides bounded by the edge of the unit square are of the same size as the interior boxes. This condition yields the value
          $$ x=\frac{2\varphi^{m+1}}{\varphi^{m-1}-\varphi^2+\sqrt{(\varphi^{m-1}-\varphi^2)^2+8\varphi^{m+1}}}. $$
        \item If $m$ is even, choose $x$ such that those maximal empty box in the lower right (or upper left) corner of the point set of width $(x+\varphi)/L$ with two sides bounded by the edge of the unit square is of the same size as the interior boxes. This condition yields the value
          $$ x=\frac12 \left(-(\varphi^{m-2}+\varphi^{-1})+\sqrt{(\varphi^{m-2}+\varphi^{-1})^2+4(\varphi^{m}+\varphi)}\right). $$
    \end{itemize}
    In both cases $x$ as a function of $m$ is increasing and tends to $\varphi^2$. For $m\in\{5,6,7,8,9,10\}$ the point set $\widetilde{\mathcal{F}}_m'$ has a dispersion of $\frac{\varphi^{m+1}}{L^2}$, which is the area of the interior boxes as well as the area of the largest exterior box (we checked by hand that all other maximal exterior boxes are smaller). We conjecture that this dispersion formula holds for all $m\geq 5$. 
    If this were the case, the asymptotic dispersion of $\widetilde{\mathcal{F}}_m'$ would still be the same as for $\widetilde{\mathcal{F}}_m$. Of course $\frac{\varphi^{m+1}}{L^2}$ is a lower bound on $\mathrm{disp}(\widetilde{\mathcal{F}}_m')$.

      \begin{table}[h]
\centering
\begin{tabular}[h]{ |c| c| c |c|  } \hline 
  $m$ & $(F_m-1)\mathrm{disp}(\mathcal{F}_m^*)$ & $(F_m-1) \mathrm{disp}(\widetilde{\mathcal{F}}_m^*)$ & $(F_m-1)\mathrm{disp}(\widetilde{\mathcal{F}}_m'^*)$  \\ \hline
    5 & $1.44$ & $1.52786$ & $1$  \\
	6 & $1.64063$ & $1.65248$ & $1.28438$  \\ 
    7 & $1.77514$ & $1.75078$ & $1.40661$  \\
	8 & $1.81406$ & $1.80340$ & $1.57491$  \\ 
    9 & $1.88408$ & $1.83903$ & $1.66684$  \\ 
	10 & $1.92793$ & $1.85986$ & $1.74963$  \\ 
	12 & $1.97232$ & $1.88125$ & $\textit{(1.83465)}$ \\ 
	15 & $1.99345$ & $1.89132$ & $\textit{(1.87970)}$  \\ 
	25 & $1.99995$ & $1.89440$ & $\textit{(1.89431)}$  \\  
	30 & $1.99999$ & $1.89442$ & $\textit{(1.89442)}$  \\  \hline
\end{tabular}
\caption{Comparison of the dispersions of $\mathcal{F}_m^*$, $\widetilde{\mathcal{F}}_m^*$ and $\widetilde{\mathcal{F}}_m'^*$, where the asterisks indicate that we exclude the point $(0,0)$. Cursive numbers in brackets are conjectured values.}
\label{compdisp}
\end{table}

    \begin{figure}[h] \label{comparisonoflattices}
     \centering
     {\includegraphics[width=70mm]{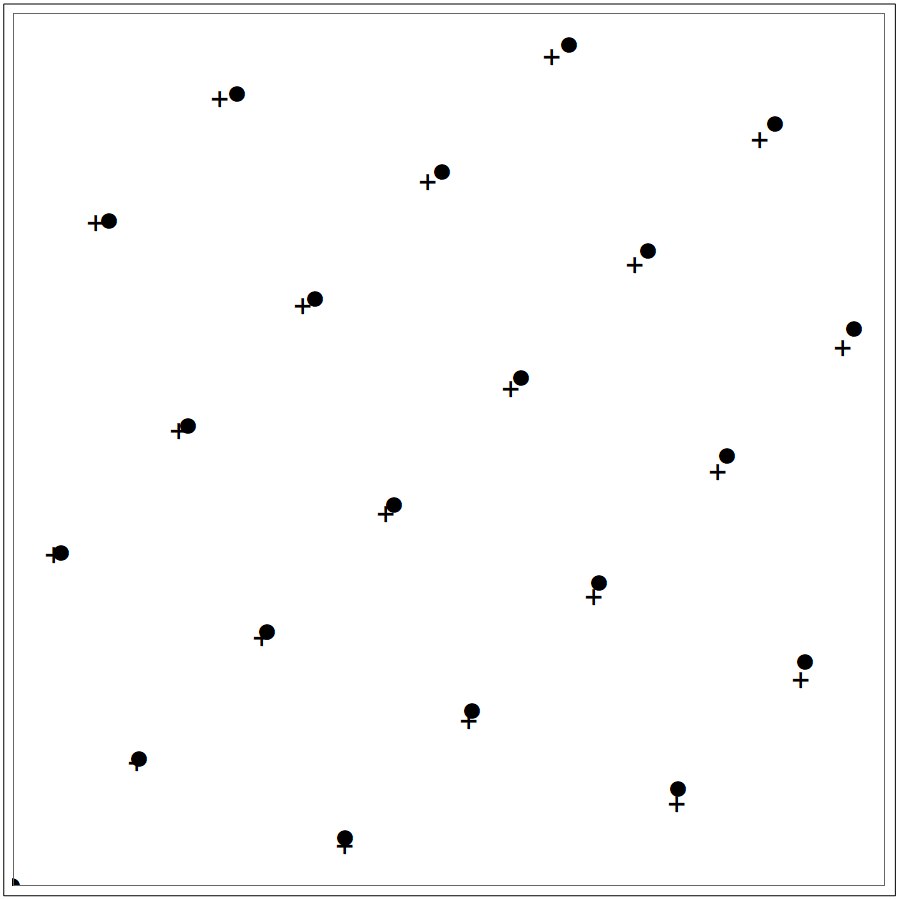}}
     \hspace{.1in}
     {\includegraphics[width=70mm]{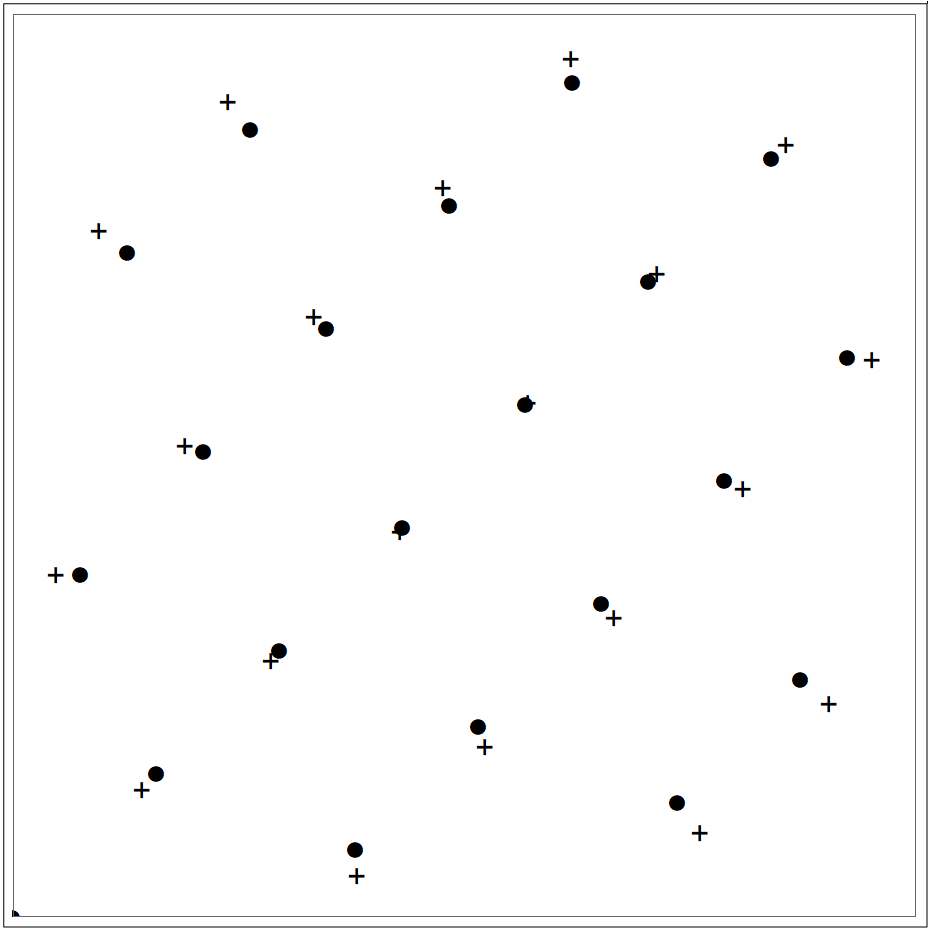}}
     \caption{A comparison of the Fibonacci lattice (+, both pictures) and $\widetilde{\mathcal{F}}_m$ (dots, left picture) and $\widetilde{\mathcal{F}}_m'$ (dots, right picture). While $\widetilde{\mathcal{F}}_m$ moves the points of the Fibonacci lattice towards the upper right corner, $\widetilde{\mathcal{F}}_m'$ moves them towards the center.}
     \label{comparison}
\end{figure}
    
    It is worthwhile to compare the dispersion of the modified Fibonacci lattices to other known constructions of point sets in the unit square with low dispersion. Krieg considered sparse grids in~\cite{Krieg}. Although they are not optimal with respect to dispersion asymptotically, these point sets have low dispersion for a small number of points. In dimension 2 sparse grids $\cP(k,2)$, where $k\in\NN$, as defined in~\cite{Krieg} have $N=2^k(k+1)$ points and a dispersion of $2^{-k-1}$. In particular, for $k=1$ this construction yields a set of 4 points with dispersion $4\,\mathrm{disp}(\cP(1,2))=1$. So does our point set $\widetilde{\mathcal{F}}_5'^*$. This value is best possible, since $4\,\mathrm{disp}(\cP_4)\geq 1$ for all $4$-element point sets $\cP_4$ in the unit square (see~\cite[Lemma 4]{Dumi}). (Note that this fact on the optimal dispersion of $4$-element point sets yields the lower bound~\eqref{lower54}. Therefore finding the minimal dispersion of $N$-element point sets for any $N\geq 5$ can probably improve this lower bound.) The grid $\cP(2,2)$ has 12 points (as many as $\widetilde{\mathcal{F}}_7'^*$) and a dispersion of $12\,\mathrm{disp}(\cP(2,2))=1.5$, which is lower than the dispersion of $\widetilde{\mathcal{F}}_7$, but already larger than the dispersion of $\widetilde{\mathcal{F}}_7^*$. The grid $\cP(3,2)$ has 32 points (one less than $\widetilde{\mathcal{F}}_9'^*$) and a dispersion of $32\,\mathrm{disp}(\cP(3,2))=2$, which is slightly larger than the dispersion of $\mathcal{F}_9^*$ already.

As mentioned above, the Fibonacci lattice $\mathcal{F}_m$ has the smallest possible torus dispersion for a point set with $F_m$ elements. It is also conjectured to have smallest possible periodic (torus) $L_2$ discrepancy, since for $m\leq 7$ this is known to be the case (see~\cite{HOe}). This discrepancy notion is defined with respect to periodic boxes as test sets. The periodic $L_2$ discrepancy is defined as  $$L_{2,N}^{\mathrm{per}}(\cP):=\left(\int_{[0,1]^2}\int_{[0,1]^2}\left|A(B(\bsx,\bsy),\cP)-N\lambda(B(\bsx,\bsy))\right|^2\rd \bsx\rd\bsy\right)^{\frac{1}{2}},  $$
where for any measurable subset $M$ of $[0,1]^2$ we define $$A(M,\cP):=|\{n \in \{0,1,\ldots,N-1\} \ : \ \bsx_n \in M\}|,$$ i.e., the number of elements from $\cP$ that belong to the set $M$. \\
The $L_2$ discrepancy of a point set can also be measured with respect to non-periodic boxes $B\in\mathcal{B}$. We call this notion extreme $L_2$ discrepancy and define it formally as
$$  L_{2,N}^{\mathrm{extr}}(\cP):=\left(\int_{[0,1]^2}\int_{[0,1]^2,\, \bsx\leq \bsy}\left|A([\bsx,\bsy),\cP)-N\lambda([\bsx,\bsy))\right|^2\rd \bsx\rd\bsy\right)^{\frac{1}{2}},  $$
where for $\bsx=(x_1,x_2)$ and $\bsy=(y_1,y_2)$ we set $[\bsx,\bsy)=[x_1,y_1)\times [x_2,y_2)$ and mean by $\bsx \leq \bsy$ that $0\leq x_1\leq y_1\leq 1$ and $0\leq x_2\leq y_2\leq 1$. \\
 The best-studied kind of $L_2$ discrepancy is the (standard) $L_2$ discrepancy, where the test sets are boxes anchored in the origin. The formal definition is
$$   L_{2,N}(\cP):=\left(\int_{[0,1]^2}\left|A([\bszero,\bst),\cP)-N\lambda([\bszero,\bst))\right|^2\rd \bst\right)^{\frac{1}{2}},  $$
where for $\bst=(t_1,t_2)\in [0,1]^2$ we set $[\bszero,\bst)=[0,t_1)\times [0,t_2)$ with area $\lambda([\bszero,\bst))=t_1t_2$. \\
It seems reasonable to compare point sets with small dispersion with respect to their standard and extreme $L_2$ discrepancies, as in all cases the test sets are non-periodic boxes. We calculated these quantities for the point sets $\mathcal{F}_m$ and $\widetilde{\mathcal{F}}_m$ for $m\in\{6,7,\dots,14\}$ using the explicit formulas stated in~\cite[Remark 14]{HKP}.
We also considered symmetrized variants of these point sets, where for a point set $\cP$ in the unit square we define its symmetrized version by $\cP^s:=\cP\cup \{(x,1-y): (x,y)\in\cP\}$. The numerical results can be found in Table~\ref{extreme}. We observe that $\widetilde{\mathcal{F}}_m$ and its symmetrized version $\widetilde{\mathcal{F}}_m^s$ outperform the (symmetrized) Fibonacci lattice for small values of $m$ indeed, whereas we found that this is not the case for $\widetilde{\mathcal{F}}_m'$ and its symmetrized variant. So far the symmetrized Fibonacci lattice has the lowest asymptotic $L_2$ discrepancy known among all point sets in the unit square (see~\cite{bilyk2}), but our numerical results indicate that our newly introduced point sets could have significantly lower $L_2$ discrepancy.
For more information on the three notions of $L_2$ discrepancy we introduced above and corresponding results on the (symmetrized) Fibonacci lattice we refer to~\cite{bilyk2, bilyk, HKP}.

  \begin{table}[ht]
\centering
\begin{tabular}[h]{ |c|| c|c||c|c||c|c|  } \hline 
  $m$ & $L_{2,F_m}^{\mathrm{extr}}(\mathcal{F}_m)$ & $L_{2,F_m}^{\mathrm{extr}}(\widetilde{\mathcal{F}}_m)$ & $L_{2,F_m}(\mathcal{F}_m)$ & $L_{2,F_m}(\widetilde{\mathcal{F}}_m)$ & $L_{2,2F_m}(\mathcal{F}_m^s)$ & $L_{2,2F_m}(\widetilde{\mathcal{F}}_m^s)$ \\ \hline
	6 & $0.23199$ & $0.22865$  & $0.89146$ & $0.77149$ & $0.74419$ & $0.60696$ \\ 
	7 & $0.24735$ & $0.24522$   & $0.84771$ & $0.68350$ & $0.75282$ & $0.65792$ \\ 
	8 & $0.26229$ & $0.26002$    & $0.91680$ & $0.79646$ & $0.76290$ & $0.62243$ \\ 
	9 & $0.27608$ & $0.27408$    & $0.87375$ & $0.70560$ & $0.77265$ & $0.67941$\\ 
	10 & $0.28931$ & $0.28737$   & $0.94094$ & $0.82293$ & $0.78221$ & $0.64240$ \\ 
	11 & $0.30191$ & $0.30007$    & $0.89899$ & $0.73295$ & $0.79167$ & $0.70095$ \\ 
	12 & $0.31402$ & $0.31225$   & $0.96441$ & $0.84927$  & $0.80104$ & $0.66386$\\ 
	13 & $0.32567$ & $0.32396$     & $0.92353$ & $0.76150$ & $0.81029$ & $0.72190$  \\ 
	14 & $0.33692$ & $0.33526$   & $0.98733$ & $0.87505$  & $0.81944$ & $0.68544$\\ \hline
\end{tabular}
\caption{Comparison of $L_2$ discrepancies of $\mathcal{F}_m$, $\widetilde{\mathcal{F}}_m$ and their symmetrized versions.  We also tested $\widetilde{\mathcal{F}}_m'$ and $\widetilde{\mathcal{F}}_m'^s$ with respect to their $L_2$ discrepancies, but they perform worse than the Fibonacci lattice despite the low dispersion of $\widetilde{\mathcal{F}}_m'$.}
\label{extreme}
\end{table}

\section{Conclusion and unanswered questions} \label{conclusion}

Our results show that
\begin{equation*}
    \liminf_{N\to\infty} N\disp(N,2)\in \left[\frac54,\frac{\varphi^3}{\sqrt{5}}\right].
\end{equation*}
Until now the right limit of this interval was 2. The exact value of $\liminf_{N\to\infty} N\disp(N,2)$ remains unknown, but we conjecture that $\varphi^3/\sqrt{5}$ is optimal. We believe that, given the success of the Fibonacci lattice with respect to other measures of uniformity, especially the torus dispersion, it is reasonable to assume that an optimal point set with respect to dispersion would be similar to such lattices. Moreover, the property that every maximal interior empty box amidst the points of $\widetilde{\mathcal F}_m$ has the same area as its dispersion seems too nice to not lead to the best possible constant. Finally, the optimality of the dispersion of $\widetilde{\mathcal F}_5'^*$ supports our conjecture.

It might be interesting to investigate whether the point sets $\widetilde{\mathcal{F}}_m$ and $\widetilde{\mathcal{F}}_m^s$ show good distribution properties with respect to other measures of irregularities of distribution too. Numerical experiments show that the extreme and standard $L_2$ discrepancy of $\widetilde{\mathcal{F}}_m$ is smaller than the respective $L_2$ discrepancies of the Fibonacci lattice for $m\in\{6,7,\dots,14\}$ (see Table~\ref{extreme}). The same holds for the standard $L_2$ discrepancy of the symmetrized lattices.  We do not know if this is the case for general $m$ and if the difference is significant in the sense that $$ \lim_{m\to\infty} \frac{L_{2,F_m}^{\mathrm{extr}}(\widetilde{\mathcal{F}}_m)}{L_{2,F_m}^{\mathrm{extr}}(\mathcal{F}_m)}<1  \text{, \,} \lim_{m\to\infty} \frac{L_{2,F_m}(\widetilde{\mathcal{F}}_m)}{L_{2,F_m}(\mathcal{F}_m)}<1 \text{\, and/or \,}  \lim_{m\to\infty} \frac{L_{2,2F_m}(\widetilde{\mathcal{F}}_m^s)}{L_{2,2F_m}(\mathcal{F}_m^s)}<1.$$

Finally we would like to mention an (unpublished) result by Thomas Lachmann which he found several years ago and is closely related to the content of this paper. He communicated this result to us when the preparation of the current paper was in its final stages. Lachmann considered a rotated variant of the infinite grid $\ZZ^2$. More precisely, he considered the grid
  $ \Gamma:=\left\{M\cdot\bsz: \bsz \in \ZZ^2\right\}$, 
  where $$M=\frac{1}{\sqrt{\varphi^2+1}}\begin{pmatrix}
\varphi & 1  \\
-1  & \varphi 
\end{pmatrix}. $$
He found that every maximal empty axes-parallel box amidst $\Gamma$ has same area $\varphi^4/(\varphi^2+1)$. Now choose a real number $R\geq 2$ and define $\Gamma(R):=\frac{1}{R}(\Gamma\cap [0,R]^2)$. Then the dispersion of the set $\Gamma(R)\subseteq [0,1]^2$ is $$\mathrm{disp}\left(\Gamma(R)\right)=\frac{1}{R^2}\frac{\varphi^4}{\varphi^2+1}.$$
Since $|\Gamma(R)|\approx \lambda([0,R]^2) \det(M)=R^2$, we have $\lim_{R\to\infty}\frac{|\Gamma(R)|}{R^2}=1$ and therefore $$\lim_{R\to\infty}|\Gamma(R)|\mathrm{disp}\left(\Gamma(R)\right)=\frac{\varphi^4}{\varphi^2+1}=\frac{\varphi^3}{\sqrt{5}}.$$


\begin{thebibliography}{}
%
%
\bibitem{Aist}
C.Aistleitner, A. Hinrichs and D. Rudolf, On the size of the largest empty box amidst a point set, Discrete Appl. Math., 230(1) (2017) 146--150. 

\bibitem{bilyk2} D. Bilyk, V. N. Temlyakov and R. Yu: Fibonacci sets and symmetrization in discrepancy theory. J. Complexity 28(1) (2012) 18--36.

\bibitem{bilyk} D. Bilyk, V. N. Temlyakov and R. Yu: The $L_2$ Discrepancy of Two-Dimensional Lattices. Recent Advances in Harmonic Analysis and Applications, pp. 63--77, Springer Proc. Math. Stat., 25, Springer, New York, 2013. 

\bibitem{Bren}
S. Breneis and A. Hinrichs, Fibonacci lattices have minimal dispersion on the torus. In: D. Bilyk, J. Dick and F. Pillichshammer, editors, Discrepancy theory, De Gruyter, 2020, pp. 117--132.

\bibitem{Dumi}
A. Dumitrescu and M. Jiang, On the largest empty axes-parallel box amidst $n$ points, Algorithmica, 66(2) (2013) 225--248.

\bibitem{Fiat} A. Fiat and A. Shamir: How to find a battleship. Networks, 19(3) (1989) 361--371.

\bibitem{HOe} A. Hinrichs and J. Oettershagen: Optimal point sets for quasi-Monte Carlo integration of bivariate periodic functions with bounded mixed derivatives. Monte Carlo and quasi-Monte Carlo methods, pp. 385--405, Springer Proc. Math. Stat., 163, Springer, [Cham], 2016. 



\bibitem{HKP}
A. Hinrichs, R. Kritzinger and F. Pillichshammer: Extreme and periodic $L_2$ discrepancy of plane point sets. Preprint, 2020, Available at https://arxiv.org/abs/2005.09933.

\bibitem{Krieg} D. Krieg:
On the dispersion of sparse grids, J. Complexity, 45 (2018) 115--119.

\bibitem{Lin} N. Linial, M. Luby, M. Saks and D. Zuckerman: Efficient construction of a
small hitting set for combinatorial rectangles in high dimension. Combinatorica 17(2) (1997) 215--234.

\bibitem{Lit} A. Litvak:
A remark on the minimal dispersion. Preprint, 2020, Available at https://arxiv.org/abs/2005.12243.

\bibitem{Rote}
G. Rote and R. F. Tichy, Quasi-Monte Carlo methods and the dispersion of point
sequences, Math. Comput. Modelling,  23(8--9)  (1996) 9--23.

\bibitem{rudo} D. Rudolf: An upper bound of the minimal dispersion via delta covers. Contemporary Computational Mathematics -- a Celebration of the 80th Birthday of Ian Sloan. Springer-Verlag, 2018.

\bibitem{Sos} J. Sosnovec, A note on minimal dispersion of point sets in the unit cube. European Journal of Combinatorics, 69 (2018) 255--259.

\bibitem{Ullr} M. Ullrich, A lower bound for the dispersion on the torus. Math. Comput. Simulation, 143 (2018) 186--190.

\bibitem{Ullr2}
M. Ullrich and J. Vyb\'{i}ral, An upper bound on the minimal dispersion, J. Complexity 45 (2018) 120--126.

\bibitem{Ullr3}
M. Ullrich and J. Vyb\'{i}ral, Deterministic constructions of high-dimensional sets with small dispersion. Preprint, 2019, Available at https://arxiv.org/abs/1901.06702.

\end{thebibliography}
\end{document}